\documentclass{amsart}
\usepackage{amssymb,amsmath,tikz}
\usepackage[utf8]{inputenc}
\newcommand{\tors}{{\sf {tors}}~\!}
\newcommand{\torss}{{\sf {tors}}}
\newcommand{\tf}{{\sf {tf}}~\!}
\renewcommand{\mod}{{\sf {mod}}~\!}
\newcommand{\ind}{{\sf {ind}}~\!}
\newcommand{\im}{{\sf im}}
\newcommand{\C}{\mathcal C}
\newcommand{\T}{\mathcal T}
\newcommand{\F}{\mathcal F}
\newcommand{\U}{\mathcal U}
\newcommand{\Y}{\mathcal Y}
\newcommand{\Z}{\mathcal Z}
\newcommand{\X}{\mathcal X}
\newcommand{\V}{\mathcal V}
\newcommand{\Hom}{\operatorname{Hom}}
\newcommand{\uperp}{{}^\perp}
\newcommand{\onto}{\twoheadrightarrow}
\newcommand{\into}{\hookrightarrow}
\newtheorem{thm}{Theorem}[section]
\newtheorem{proposition}{Proposition}[section]
\newtheorem{cor}{Corollary}[section]
\newtheorem{lemma}{Lemma}[section]
\newtheorem{example}{Example}[section]
\newtheorem{question}{Question}[section]
\newcommand{\cji}{\operatorname{Ji}^c}
\newcommand{\cmi}{\operatorname{Mi}^c}
\newcommand{\phibar}{\overline\phi}
\renewcommand{\to}{\mapsto}
\newcommand{\sto}{\rightarrow}
\newcommand{\br}{{\sf br}~\!}
\newcommand{\Gen}{\operatorname{Gen}}
\usetikzlibrary{decorations.pathreplacing}
\usetikzlibrary{arrows}
\DeclareFontFamily{U}{wncy}{}
    \DeclareFontShape{U}{wncy}{m}{n}{<->wncyr10}{}
    \DeclareSymbolFont{mcy}{U}{wncy}{m}{n}
    \DeclareMathSymbol{\Br}{\mathord}{mcy}{"58}

\title{An introduction to the lattice of torsion classes}
\author{Hugh Thomas}

\address{Lacim, UQAM, CP 8888, Succursale Centre-ville, Montréal, QC, H3C 3P8 Canada}

\begin{document}

\begin{abstract}In this expository note, I present some of the key features of the lattice of torsion classes of a finite-dimensional algebra, focussing in particular on its complete semidistributivity and consequences thereof. This is intended to serve as an introduction to recent work by Barnard--Carroll--Zhu and Demonet--Iyama--Reading--Reiten--Thomas.\end{abstract}
\maketitle

Let $A$ be a finite-dimensional algebra over a field $k$. We write
$\mod A$ for the category of finite-dimensional left $A$-modules.
There is a class of subcategories of $\mod A$ which are called torsion classes.
The torsion classes, ordered by inclusion, form a poset which we denote
$\tors A$.
This poset is in fact a lattice, and its lattice-theoretic properties  have
recently been the focus of
some attention, as in \cite{J, GM, BCZ, DIRRT, AP}.

In this note I will present some of the interesting features of these lattices.
The proofs in this note are self-contained except
for the final section, where we present without proof an application of
these ideas to the study of finite semidistributive lattices from
\cite{RST}. This note is intended as a gentle introduction to the subject.
No results in this note are new. The presentation is, of course, novel in
some respects,
and I hope that it is helpful as an introduction to the subject.


Let me now quickly summarize the contents of this note. Terms which are
undefined here will be introduced later where they logically fit.
In addition to presenting the easy explanation that $\tors A$ is a lattice,
I will prove the result of Barnard, Carroll, and Zhu \cite{BCZ} that the
completely join irreducible elements of $\tors A$ are in bijection with the bricks
of $A$.
I will show that $\tors A$ is completely semidistributive. I will not take the
most direct route to this result, but rather spend some time developing independently
properties of $\tors A$ and corresponding properties of semidistributive
lattices, in an attempt to illuminate how 
semidistributivity gives us a
helpful perspective through which to view the combinatorics of $\tors A$.
I will show that $\tors A$ is weakly atomic. We will then see how
an algebra
quotient induces a lattice quotient map between the corresponding lattices of torsion classes, and study this lattice quotient.
In the final section, I will present (without proof) a construction of
finite semidistributive lattices developed in \cite{RST}, and inspired
by the study of lattices of torsion classes.


\section{Definition of torsion classes}
For the elementary material in this section and the two following, 
a further
reference is \cite[Chapter VI]{ASS}.

A torsion class in $\mod A$ is a subcategory $\mathcal T$ of $\mod A$
which is \begin{itemize}
\item closed under quotients (i.e., $Y\in \mathcal T$ and $Y\onto Z$ implies
  $Z\in \mathcal T$).
\item closed under extensions (i.e., $X,Z \in \mathcal T$ and
  $0\rightarrow X \rightarrow Y \rightarrow Z \rightarrow 0$ implies
  $Y\in \mathcal T$.
\end{itemize}

I should clarify that for me a subcategory is always full, closed under
direct sums, direct summands, and isomorphisms. In other words, a
subcategory of $\mod A$
can be specified as the direct sums of copies of some subset of
the indecomposable modules of $A$.

We write $\tors A$ for the set of torsion classes of $\mod A$, and we think
of it as a poset ordered by inclusion.

\begin{example}[Type $A_2$]  Our quiver $Q$ is $1\leftarrow 2$, and the algebra is the path algebra  $A=kQ$.
The category $\mod A$
has three indecomposable objects $S_1, P_2, S_2$, which I denote by
their dimension vectors as $[10]$, $[11]$, and $[01]$, respectively.

The torsion classes are as follows, where the angle brackets denote additive
hull.
$$
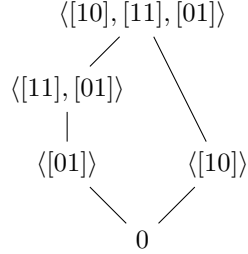
\begin{tikzpicture}
  \node (a) at (0,0) {$\langle [10],[11],[01] \rangle$};
  \node (b) at (-1,-1) {$\langle [11],[01]\rangle$};
  \node (c) at (-1,-2) {$\langle [01]\rangle$};
  \node (d) at (0,-3) {$0$};
  \node (e) at (1,-2) {$\langle [10]\rangle$};

\draw (a) -- (b) -- (c) -- (d) -- (e) -- (a);
\end{tikzpicture}$$
\end{example} 
 
\begin{example}[Type $A_n$] For an example in type $A_n$, where $Q$ is
  $1\leftarrow \dots \leftarrow n$, see \cite{K}. \end{example}

\begin{example}[Kronecker quiver]  \label{ex3}
  Let $k$ be algebraically closed. Let
  $Q$ be the quiver
  $\begin{tikzpicture}[baseline=-1.1mm] \node (a) at (0,0) {$1$};
    \node (b) at (1,0) {$2$};
    \draw[->] ([yshift=.5mm]b.west) -- ([yshift=.5mm]a.east);
    \draw[->] ([yshift=-.5mm]b.west) -- ([yshift=-.5mm]a.east); \end{tikzpicture}$
  and let $A=kQ$.

  The AR quiver is displayed in Figure \ref{figb}, where I write $[ab]$ for an indecomposable
  module with dimension vector $(a,b)$.
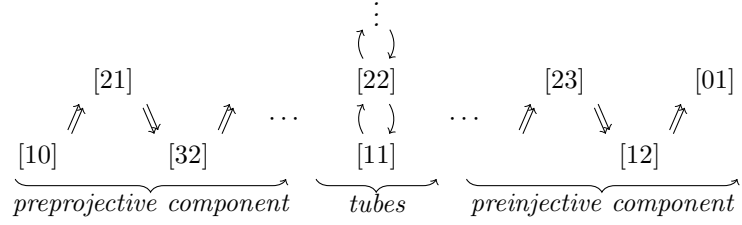
\begin{figure}
  $$\begin{tikzpicture}[->]

    \node (a) at (0,0) {$[10]$};
    \node (b) at (1,1) {$[21]$};
    \node (c) at (2,0) {$[32]$};
    \node (d) at (3,1) {\phantom{$[43]$}};
\node (x) at (3.3,.5) {$\cdots$};
    
    \node (e) at (4.5,0) {$[11]$};
    \node (f) at (4.5,1) {$[22]$};
    \node (g) at (4.5,2) {$\vdots$};

    \draw ([xshift=.5mm]a.north east) -- ([yshift=-.5mm]b.south west);
    \draw ([yshift=.5mm]a.north east) -- ([xshift=-.5mm]b.south west);
    \draw ([xshift=.5mm]b.south east) -- ([yshift=.5mm]c.north west);
    \draw ([yshift=-.5mm]b.south east) -- ([xshift=-.5mm]c.north west);
    \draw ([xshift=.5mm]c.north east) -- ([yshift=-.5mm]d.south west);
    \draw ([yshift=.5mm]c.north east) -- ([xshift=-.5mm]d.south west);

    \draw  (e) to [bend left=30] (f);
    \draw  (f) to [bend left=30](g);
    \draw (g) to [bend left=30] (f);
    \draw (f) to [bend left=30] (e);

    \node (h) at (6,0) {\phantom{$[10]$}};
    \node (i) at (7,1) {$[23]$};
    \node (j) at (8,0) {$[12]$};
    \node (k) at (9,1) {$[01]$};
\node (l) at (5.7,.5) {$\cdots$};

    \draw ([xshift=.5mm]h.north east) -- ([yshift=-.5mm]i.south west);
    \draw ([yshift=.5mm]h.north east) -- ([xshift=-.5mm]i.south west);
    \draw ([xshift=.5mm]i.south east) -- ([yshift=.5mm]j.north west);
    \draw ([yshift=-.5mm]i.south east) -- ([xshift=-.5mm]j.north west);
    \draw ([xshift=.5mm]j.north east) -- ([yshift=-.5mm]k.south west);
    \draw ([yshift=.5mm]j.north east) -- ([xshift=-.5mm]k.south west);

    \draw[decorate,decoration={brace,mirror,amplitude=5pt}] (-.3,-.3) -- node[yshift=-2pt][below]{preprojective component} (3.3,-.3);
    \draw[decorate,decoration={brace,mirror,amplitude=5pt}] (3.7,-.3) -- node [yshift=-2pt][below] {tubes} (5.3,-.3);
    \draw[decorate,decoration={brace,mirror,amplitude=5pt}] (5.7,-.3) -- node[yshift=-2pt][below]{preinjective component} (9.3,-.3);
    
  \end{tikzpicture}$$
  \caption{\label{figb} The AR quiver of the path algebra of the Kronecker quiver.}
\end{figure}
  
  The tubes are indexed by points in $\mathbb P^1(k)=k\bigcup\{\infty\}$;
  they each look the same. 
  The torsion classes consist of the additive hull of each of the following sets:\begin{itemize}
  \item  any final part of the preinjective component,
  \item all preinjectives and a subset of the tubes,
  \item all preinjectives, all tubes, and a final part of the
    preprojectives,
  \item $S_1=[10]$.\end{itemize}

  An image of the lattice of torsion classes is displayed in Figure
  \ref{figa}. There, $\mathcal I$ denotes the
  preinjective component, and $\mathcal R_x$ denote the tube corresponding to
  $x\in \mathbb P^1(k)$. I write $\ind \mod A$ for the set of indecomposable
  $A$-modules.
  
  The interval between $\mathcal I$ and $\langle \mathcal I, \bigcup_{x\in\mathbb P^1(k)} \mathcal R_x\rangle$ is isomorphic to the Boolean lattice of all subsets of $\mathbb P^1(k)$, ordered by inclusion. 

  \begin{figure}
  $$\begin{tikzpicture}

    \node (a) at (-1,-0.3) {$0$};
    \node (b) at (1,0.7) {$\langle [01]\rangle$};
    \node (c) at (1,2) {$\langle [12],[01]\rangle$};
    \node (d) at (1,3) {$\vdots$};
    \node (e) at (1,4) {$\mathcal I$};
    \node (h) at (-1,5) {$\langle \mathcal I, \mathcal R_{-1}\rangle$};
    \node (hh) at (.1,5) {$\cdots$};
    \node (f) at (1,5) {$\langle \mathcal I,\mathcal R_0\rangle$};
    \node (ff) at (2,5) {$\cdots$};
    \node (g) at (3,5) {$\langle \mathcal I,\mathcal R_1\rangle$};
    \node (i) at (1,7) {$\langle \mathcal I,\bigcup_{x\in\mathbb P^1(k)} \mathcal R_x\rangle$};
    \node (j) at (1,6) {\phantom{$\mathcal I,\mathcal R_{-1}\rangle$}};
      \node (k) at (-1,6) {\phantom{$\langle\mathcal I,\mathcal R_{-1}\rangle$}};
      \node (l) at (3,6) {\phantom{$\langle\mathcal I,\mathcal R_1\rangle$}};
      \node (kk) at (0,6) {$\cdots$};
      \node (jj) at (2,6) {$\cdots$};
      \node (ll) at (-2,5) {$\cdots$};
      \node (mm) at (4,5) {$\cdots$};
      \node (nn) at (3,6) {$\cdots$};
      \node (pp) at (-1,6) {$\cdots$};
      \node (qq) at (1,6) {$\cdots$};
    \draw (a)-- node[below right]{\color{blue}$[01]$}(b)--node[right]{\color{blue}$[12]$}(c)--(d);
    \draw ([yshift=-2mm]d.north)--(e);
    \draw (e) -- (f);
    \draw (e) -- (g);
    \draw (e) -- (h);
    \draw (i) -- (j);
    \draw (i) -- (k);
    \draw (i) --(l);
    
    \draw[decorate,decoration={brace,mirror,amplitude=5pt}] (4.5,4)--node[xshift=2pt][right]{Boolean lattice}(4.5,7);
    \node (m) at (1,8) {$\vdots$};
    \node (n) at (1,9) {$\langle \ind \mod A \setminus\{[10],[21]\}\rangle$};
    \node (o) at (1,10.3) {$\langle \ind \mod A \setminus \{[10]\}\rangle$};
    \node (p) at (-1,11.3) {$\mod A$};
    \node (q) at (-3.5,5.5) {$\langle [10]\rangle$};

    \draw (a) -- node[left]{\color{blue}$[10]$}(q) -- node[left]{\color{blue}$[01]$}(p);
    \draw (p) -- node[above right]{\color{blue}$[10]$}(o) --node[right]{\color{blue}$[21]$}(n)-- ([yshift=-2mm] m.north);
    \draw (m)--(i);
    \end{tikzpicture}$$
    \caption{\label{figa}The lattice of torsion classes for $A$ the path algebra of the Kronecker quiver. The labels on the edges should be ignored for now; they are the brick labelling $\hat\gamma$ defined in Section \ref{cons}.}
  \end{figure}
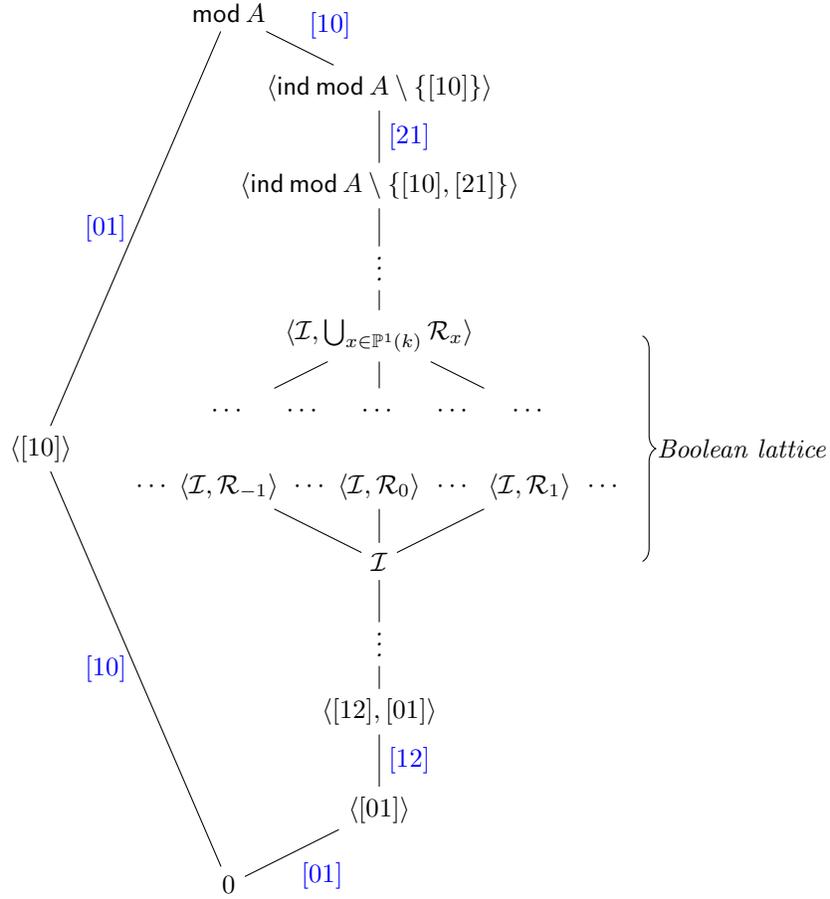
  \end{example}

\section{Specifying a torsion class}

In general, how can we specify a torsion class? For $\mathcal C$ a subcategory of
$\mod A$, define $T(\mathcal C)$ to be the subcategory whose
modules are filtered by quotients of objects from $\mathcal C$. That is
to say $M\in T(\mathcal C)$ if and only if $M$ admits a filtration
$0=M_0 \subset M_1 \subset \dots \subset M_r=M$ with $M_i/M_{i-1}$ a quotient of an object
of $\mathcal C$ for all $i$.

\begin{proposition}\label{p1} For $\mathcal C$ an arbitrary subcategory,
  $T(\mathcal C)$ is the smallest torsion class
containing all the objects from $\mathcal C$. \end{proposition}

\begin{proof} Suppose that $M\in T(\mathcal C)$, so that we have a
filtration $0=M_0 \subset M_1 \subset \dots \subset M_r=M$ with $M_i/M_{i-1}$ a quotient
of an object of $\mathcal C$ for all $i$. Consider some quotient of
$M$, say $N=M/L$.
Then define $N_i=(M_i+L)/L$, which forms a filtration of $N$. We see that
$N_i/N_{i-1}$ is a quotient of $M_i/M_{i-1}$, and therefore a quotient of an
object of $\mathcal C$. This shows that $N\in T(\mathcal C)$.

Suppose next that we have two modules $M$ and $N$, both in $T(\mathcal C)$,
and an extension
$$0 \rightarrow M \rightarrow E \rightarrow N \rightarrow 0.$$

Now $E$ has a two-step filtration $0 \subset M \subset E$, with $E/M \simeq N$,
and we can refine the two steps of the filtration to filtrations of
$M$ and $N$ with subquotients being quotients of $\mathcal C$, since we know
such filtrations exist. This shows
that $E\in T(\mathcal C)$.
It follows that $T(\mathcal C)$ satisfies the two defining properties, and
is therefore a torsion class.

$T(\mathcal C)$ is the smallest torsion class containing $\mathcal C$ because any element
of $T(\mathcal C)$ is an iterated extension of quotients of $\mathcal C$, which
must be in any torsion class containing $\mathcal C$. 
\end{proof}

We now consider a second way to specify a torsion class. 
For $\mathcal C$ a subcategory of $\mod A$, define
$$\uperp \mathcal C=\{ X\in \mod A \mid \Hom(X,Y)=0 \textrm{ for all }
Y \in \mathcal C\}.$$

\begin{proposition}\label{p2} For $\mathcal C$ an arbitrary subcategory, $\uperp \mathcal C$ is a torsion class.
\end{proposition}

\begin{proof} Let $M\in\uperp \mathcal C$. Let $N$ be a quotient of $M$.
Since there are no non-zero morphisms from $M$ into any object of
$\mathcal C$, the same holds for $N$, so $N\in \uperp \mathcal C$.

Suppose now that we have $M$ and $N$ in $\uperp \mathcal C$, and an
extension:
$$0 \rightarrow M \rightarrow E \rightarrow N \rightarrow 0.$$
For any $Y\in \mathcal C$, we have $\Hom(M,Y)=0$ and $\Hom(N,Y)=0$, and
it follows from the left exactness of the $\Hom$ functor that
$\Hom(E,Y)=0$ as well. We deduce that $E \in \uperp \mathcal C$.

$\uperp \mathcal C$ satisfies the two defining conditions, and is therefore
a torsion class.
\end{proof}

\section{Torsion classes and torsion free classes}

There is a dual notion to that of torsion class, namely that of torsion free
class. A torsion free class in $\mod A$ is a subcategory closed under
submodules and extensions. 
We write $\tf A$ for the torsion free classes of $A$, and we think of it
as a poset ordered by inclusion.

As one should expect, in the setting of finite-dimensional algebras in
which we work, the theory of torsion free classes is completely parallel
to the theory of torsion classes. For $\C$ a subcategory of $\mod A$, 
define $F(\mathcal C)$ to be the
subcategory of $\mod A$ consisting of all modules filtered by submodules of
modules from $\mathcal C$. Then $F(\C)$ is the smallest torsion free class
containing $\mathcal C$. We can also define
$$\mathcal C^\perp=\{Y\in \mod A\mid \Hom(X,Y)=0 \textrm{ for all }
X\in \C\}.$$
One easily checks that for any subcategory
$\mathcal C$, the subcategory $\mathcal C^\perp$ is a torsion free class.

\begin{proposition}
Let $\T$ be a torsion class, and let $X\in\mod A$. There is a maximum
submodule of $X$ contained in $\T$.
\end{proposition}

\begin{proof} If $M$ and $N$ are submodules
of $X$, then we have a short exact sequence
$$ 0 \rightarrow M \rightarrow M+N \rightarrow N/(N\cap M)\rightarrow 0$$
If $N$ and $M$ are both in $\mathcal T$, it follows that $M+N$ is also.
Because $X$ is finite-dimensional by assumption, it therefore has a
maximum submodule contained in $\mathcal T$.
\end{proof}

We denote this maximum submodule by
$t_\T X$. 

\begin{proposition} $X/t_\T X$ lies in $\T^\perp$.
\end{proposition} 

\begin{proof} Suppose there were a non-zero map $f$ from some $M\in \T$ to
$X/t_\T X$. Then $\im f$ is a quotient of $M$, and therefore itself
in $\T$. The preimage of $\im f$ in $X$ is then an extension of $\im f$ by
$t_\T X$, and is therefore also in $\T$, contradicting the definition of
 $t_\T X$. \end{proof}

For any $X$ in $\mod A$, we  now have a short exact sequence:\begin{equation*}
  0 \rightarrow t_\T X \rightarrow X \rightarrow X/t_\T X\rightarrow 0.\tag{$*$}\end{equation*}
with the lefthand term in $\T$ and the righthand term in $\T^\perp$.

\begin{proposition} For $X\in \mod A$, any short exact sequence of the form
$$ 0\rightarrow X' \rightarrow X \rightarrow X'' \rightarrow 0$$
  with $X'$ in $\T$ and $X''\in \T^\perp$ is isomorphic to
  $(*)$. \end{proposition} 

\begin{proof} Viewing $X'$ as a submodule of $X$, it must be contained in
  $t_\T X$. If the containment were strict, then 
$X''$ would not lie in $\mathcal T^\perp$. The result follows. \end{proof}

We can now prove the following theorem:

\begin{thm}\label{rev} The map $\T \to \T^\perp$ is an inclusion-reversing bijection
from torsion classes to torsion free classes. Its inverse is given by the
map $\F \to \uperp\F$. \end{thm}

\begin{proof} Let $\T$ be a torsion class. We already pointed out that $\T^\perp$ is torsion free. It is
easy to see that $\uperp(\T^\perp)\supseteq \T$. For the other inclusion,
suppose $X\in \uperp(\T^\perp)$. Since $X/t_\T X\in \T^\perp$, there are
no non-zero morphisms from $X$ to $X/t_\T X$. But this must mean that
$X/t_\T X=0$, so $X=t_\T X$, and $X \in \T$.

Starting with a torsion free class $\F$, we see just as easily that the
composition of the two maps in the other order is also the identity. They
are therefore inverse bijections.
It is easy to see that they are order-reversing.
\end{proof}

From the previous theorem, together with Proposition \ref{p2},
the following corollary follows:

\begin{cor} The following pairs of subcategories are the same:
\begin{itemize}
\item $\{(\T,\T^\perp)\mid \T \in \tors A\}$,
\item $\{(\uperp \F,\F)\mid \F \in \tf A\}$,
\item $\{(\mathcal X,\mathcal Y)\mid \X=\uperp \Y, \Y=\X^\perp)\}$.
\end{itemize}
\end{cor}

\section{Posets and lattices}

A possible reference for basis material on lattices is \cite{Gr}.

A poset is a partially ordered set. In a poset, we say that
$x$ covers $y$ if $x$ is greater than $y$ and there is no element
$z$ such that $x>z>y$. In this case we write $x\gtrdot y$. 

A lattice is a poset in which any two elements $x$ and $y$ have a
unique least upper bound (their ``join'') denoted $x\vee y$, and a
unique greatest lower bound (their ``meet'') denoted
$x\wedge y$.

A complete lattice is a lattice such that any subset $S$ of $L$
has a unique least upper bound, which we denote either
$\bigvee_{x\in S} x$ or $\bigvee S$, and a
unique greatest lower bound, which we denote $\bigwedge_{x\in S} x$
or $\bigwedge S$.

A finite lattice is necessarily complete. The perspective taken in this
note is that the desirable infinite generalization of finite lattices are the
complete lattices. 

A complete lattice necessarily has a minimum element $\hat 0$ (the meet
of all the elements of $L$) and similarly a
maximum element $\hat 1$. 

\section{Torsion classes form a complete lattice}

The poset $\tors A$ clearly has a meet operation given by intersection,
since the intersection of two torsion classes again satisfies the
defining properties of a torsion class. The same is true for meets
of arbitrary collections of torsion classes, for the same reason.

To see the other lattice operation, there are three approaches which all
work.
Since the left perpendicular/right perpendicular operations are
order-reversing bijections between torsion-classes and torsion-free classes,
we have that
$$\bigvee_{\T\in S} \T = \uperp \left(\bigwedge_{\T\in S} \T^\perp\right).$$
Since the $\bigwedge$ on the righthand side exists (being given by intersection), so does the $\bigvee$ on the lefthand side.

We can also define the join operation in $\tors A$
implicitly. Any poset with a maximum element and a
$\bigwedge$ also has a $\bigvee$, which can be defined as follows:
$$\bigvee_{\T\in S} \T= \bigwedge_{\{\Y\in \tors A \,\mid\, 
\,\forall \T\in S,\, \Y\supseteq\T \}} \Y$$

Finally, we can also describe the join explicitly using Proposition \ref{p1}:

$$\bigvee_{\T \in S} \T = T\left(\bigcup_{\T\in S} \T\right)$$

We therefore have the following result:

\begin{proposition} $\tors A$ is a complete lattice.
\end{proposition}

\section{join irreducible elements in lattices}

An element $x$ of a lattice $L$ is called join irreducible if it cannot
be written as the join of two elements both strictly smaller than it, and
it is also not the minimum element of the lattice.
Especially for finite lattices, the join irreducible
elements can be viewed as ``building blocks'' of the lattice.

\begin{proposition} In a finite lattice $L$, any element is the join of
the join irreducible elements below it. \end{proposition}

\begin{proof} Suppose $x\in L$ were a minimal counter-example to the
statement of the proposition.
If $x$ were join irreducible, it is obviously not a counter-example,
so suppose that
it is not join irreducible. We can therefore write $x=y\vee z$ with
$y,z<x$. By the assumption that $x$ is a minimal counter-example,
$y$ and $z$ can each be written as a join of join irreducible elements.
Joining together these two expressions, we get an expression for $x$
as a join of join irreducible elements, contradicting our assumption that
$x$ was a counter-example.
\end{proof}



The situation for infinite lattices is more complicated. It can still be
interesting to consider join irreducible elements defined as above.
However, for our purposes, the following definition is more important.
We say that $x\in L$ is
completely join irreducible if $\bigvee_{y<x}y < x$. Equivalently, there
is an element, which we denote $x_*$ such that $y<x$ if and only if
$y\leq x_*$.
Note that the minimum element $\hat 0$ is not considered to be completely join irreducible.
We write $\cji L$ for the completely join irreducible elements of $L$.

Note that for a finite lattice, $x\in L$ is join irreducible if and
only if it is completely join irreducible. However, this is not
true in infinite lattices. 
For example, consider $[0,1]$, as an interval in $\mathbb R$
with the usual order. Every element except $0$ is join irreducible, but
there are no completely join irreducible elements.
This suggests that neither of these notions is necessarily all that useful for
general infinite lattices. 
However, for the lattices we
are interested in, the notion of completely join irreducible elements will
turn out to be very important. 

Let us return to consider the torsion classes of the Kronecker quiver presented
in Example \ref{ex3}. 
Of the torsion classes in the interval isomorphic to a Boolean lattice,
the elements covering the minimum are completely join irreducible, while the
others are not. Among the other torsion classes, all are completely
join irreducible except the minimum and maximum elements. The unique torsion
class which is join irreducible but not completely join irreducible is the
one composed of all the preinjective modules, labelled $\mathcal I$ in the diagram. 
It is the join of the (infinite) set of
torsion classes generated by preinjective modules, but it is not the join of
any finite set of torsion classes strictly contained in it. 

There are also dual notions of meet irreducible and
completely meet irreducible elements of a lattice. We write $\cmi L$ for the completely meet irreducible elements of $L$.
For $m$ a completely meet irreducible element, we write $m^*$ for the unique element which covers it.

\section{Completely join irreducible torsion classes}

Recall that a module $B$ is called a brick if every non-zero endomorphism
of $B$ is invertible. A brick is necessarily indecomposable, since projection
onto a proper indecomposable summand is a non-invertible
endomorphism. 
Write $\br A$ for the $A$-modules which are bricks.

In the case of the Kronecker quiver, the bricks are the indecomposable
modules from the preprojective and preinjective components, together
with the quasi-simple module at the bottom of each tube.

In this section, we shall show an important result by Barnard--Carroll--Zhu \cite[Theorem 1.5]{BCZ}, that
there is a bijection between $\br A$ and the completely join irreducible elements of
$\tors A$.

The same result holds for $\tf A$, and, by the order-reversing
bijection between $\tf A$ and $\tors A$, the same result also holds for the
meet irreducible elements of $\tors A$ and $\tf A$. For simplicity, we
will focus our attention on $\tors A$ and its completely join irreducible
elements; everything we prove has analogues in the other settings.

The following lemma says that a torsion class is characterized by
the bricks it contains.

\begin{lemma} \label{lone} Let $\T\in \tors A$. Then $$\T=\bigvee_{B\in \T\cap \br \!A}
T(B)$$ \end{lemma}

\begin{proof} Let us write
$$\mathcal U= \bigvee_{B\in \T\cap \br \!A}
T(B)$$
Clearly, $\mathcal U \subseteq \mathcal T$. Now suppose that we have some
$X$ which is in $\mathcal T$ but not in $\mathcal U$, and among such $X$,
choose one of minimal dimension. $X$ is clearly not a brick, since otherwise
it would be contained in $\U$. Thus it has a non-zero
non-invertible endomorphism
$f$. We
get a short exact sequence:
$$ 0 \rightarrow f(X) \rightarrow X \rightarrow X/f(X) \rightarrow 0$$
Since $X\in \T$, we have $X/f(X)\in\mathcal T$, and since the dimension of
$X/f(X)$ is less than that of $X$, it follows that $X/f(X) \in \U$.

Similarly, though, since $f(X)$ is also a quotient of $X$, we know
$f(X)\in\T$ and since $f(X)$ is in fact a proper quotient of $X$, the minimality assumption on $X$ then implies that  $f(X)\in \U$. We now see that $X$ is the extension
of two objects from $\U$, so it is itself in $\U$, contrary to our
assumption.\end{proof}

We also need the following lemma due to Sota Asai.

\begin{lemma}[{\cite[Lemma 1.7(1)]{Asai}}] \label{slem} If $X\in T(B)$, then either $X$ admits a surjection onto $B$ or $\Hom(X,B)=0$.
\end{lemma}

\begin{proof} Suppose that $f\in\Hom(X,B)$ is non-zero. Since $X$ is
filtered by quotients of $B$, we can write
$0=X_0 \subseteq X_1 \subseteq \dots \subseteq X_r=X$, with
$X_i/X_{i-1}$ isomorphic to a quotient of $B$. Consider the smallest $i$
such that $f|_{X_{i}}$ is non-zero. Since $f|_{X_{i-1}}=0$, $f$ induces a map
from $X_i/X_{i-1}$ to $B$, and thus from $B$ to $B$. Since $B$ is a brick,
this map must be surjective, so $f|_{X_i}$ is surjective, and thus $f$ is
surjective.
\end{proof}

We can now prove the main result of the section:

\begin{thm}[{\cite[Theorem 1.5]{BCZ}}]\label{cjib} The map $B\to T(B)$ is a bijection from $\br A$ to
$\cji \tors A$. \end{thm}

\begin{proof}
First of all, we want to show that, for $B$ a brick,
$T(B)$ is a completely join irreducible
torsion class. This requires showing that there is a unique maximum element
among all those torsion classes strictly below $T(B)$. We claim that this
torsion class can be described as $T(B) \cap \uperp F(B)$.

Since $B\not\in \uperp F(B)$, it is clear that $T(B) \cap \uperp F(B)$ is
a torsion class strictly contained in $T(B)$.
On the other hand, any torsion class strictly contained in $T(B)$ cannot contain $B$, and thus cannot contain any module $X$
admitting a surjective map onto $B$. Thus, by Lemma~\ref{slem}, any such torsion class
must be
contained in $\uperp B$. Clearly $\uperp B \supseteq \uperp F(B)$, and the reverse inclusion follows because any element of $F(B)$ is filtered by subobjects of $B$, so if $X$ has no non-zero morphisms into $B$, it has no non-zero morphisms into any element of $F(B)$. Therefore, any torsion class properly contained in $T(B)$ is contained in $T(B) \cap \uperp F(B)$. This proves the claim, thus establishing
that $T(B)$ is a completely join irreducible torsion class. 

On the other hand, by Lemma \ref{lone}, any torsion class can be written as
the join of $T(B)$ as $B$ runs through all bricks in the torsion class. This shows that any torsion class can be written as a join of the completely join irreducible torsion classes we have already identified (those of the form $T(B)$ for $B$ a brick) so there cannot be any completely join irreducible torsion classes not of this form.

Finally, we want to check that the map from bricks to torsion classes is
injective. Suppose that $T(B)=T(B')$, for $B$ and $B'$ two bricks.
$B'$ cannot be contained in $T(B)_*$. Thus there is a surjection from $B'$ to
$B$ by Lemma \ref{slem}. Reversing the rôles of $B'$ and $B$, there is
also a surjection from $B$ to $B'$. Therefore $B$ and $B'$ must be isomorphic.
\end{proof}

This theorem is one of the key justifications for the impression that
when considering lattices of
torsion classes, it is most appropriate to think in terms of the
complete versions of lattice-theoretic phenomena. As we saw in the example
of the Kronecker quiver, there is a join irreducible torsion class
which is not completely join irreducible, namely, the additive hull of
the preinjective
component. In accordance with Theorem \ref{cjib}, it does not correspond to any brick in $\mod A$. This raises the following
interesting question:

\begin{question} Is there any way to extend Theorem \ref{cjib} to characterize
  the join irreducible but not completely join irreducible elements of
  $\tors A$? \end{question}

The proof of the following theorem is dual to the proof of Theorem \ref{cjib}.

\begin{thm} The map $B\to F(B)$ is a bijection from
$\br A$ to $\cji \tf A.$\end{thm}

Then, applying Theorem \ref{rev}, we deduce:

\begin{cor} \label{cmib} The map $B\to \uperp F(B)$ is a bijection from $\br A$ to
$\cmi \tors A$.\end{cor}

From Theorem \ref{cjib}, Corollary \ref{cmib}, and their proofs, we can say
that associated to a brick $B$, there are four torsion classes, arranged
as in the following diagram, where the join of the two torsion classes on the middle layer equals the top torsion class, and their meet equals the bottom torsion class.

$$
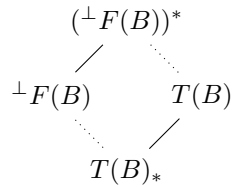
\begin{tikzpicture} \node (a) at (0,0) {$(\uperp F(B))^*$};
\node (b) at (1,-1) {$T(B)$};
\node (c) at (-1,-1) {$\uperp F(B)$};
\node (d) at (0,-2) {$T(B)_*$};
\draw[dotted] (a) -- (b);
\draw[dotted] (c) -- (d);
\draw (a)--(c);
\draw (b) --(d);
\end{tikzpicture}$$

In the diagram, the edges drawn as undashed lines are cover relations in the lattice of
torsion classes. The edges drawn using dashed lines are weak poset
relations. In particular, the torsion classes at
the endpoints of a dotted line may be equal.
Also, the pair of torsion classes not connected by a line are not
comparable in the lattice of torsion classes. We follow these conventions
in subsequent diagrams.

  \section{Parenthesis: $\tau$-tilting}

We include the following section because it makes the link to another topic
of current research related to torsion classes, which was also
presented during the spring school. A possible reference is the survey by Iyama and Reiten \cite{IR}.

  A torsion class $\mathcal T$ is called functorially finite if there is
  some $X\in \mod A$ such that $\mathcal T=\Gen(X)$, where $\Gen(X)$
  is by definition the collection of quotients of direct sums of copies of
  $X$.

  In the Kronecker case, which ones are functorially finite? Exactly
  those not in the Boolean lattice. There is no single module
  which generates the whole preinjective component and nothing more, and there is no
  single module which generates any tube without in fact being preprojective
  (and thus generating all the tubes and more).

  As shown by Adachi, Iyama, and Reiten, in the paper \cite{AIR} which introduced the topic of $\tau$-tilting theory, functorially finite torsion classes correspond bijectively to a certain class of modules
  called
  basic support $\tau$-tilting
  modules;
  the bijection from basic support $\tau$-tilting modules to torsion classes
  is $\Gen$.

  Functorially finite torsion classes need not form a lattice. There is
  nothing that guarantees that the intersection of two functorially finite
  torsion classes will be functorially finite, so in order for them to
  form a lattice anyway, there would have to be a biggest functorially finite
  torsion class contained in the intersection, and this does not always hold. Generally, for 
  hereditary algebras not of finite type, the functorially finite torsion
  classes do not form a lattice \cite {IRRTT,Ring}. Thus, for lattice-theoretic
  study, it seems preferable not to restrict to functorially finite
  torsion classes.


\section{Semidistributivity}\label{csd}

In this section we introduce the notion of semidistributivity of a lattice.
See \cite{AN,RST} for more on the subject.

A lattice $L$ is called join semidistributive if $x\vee y = x\vee y'$ implies
that $x\vee (y \wedge y')$ is also equal to both of them.
It is called completely join semidistributive if given $x\in L$ and
a set $S\subseteq L$,
such that $x\vee y=z$ for all $y\in S$, then $x\vee \bigwedge S=z$.

Join semidistributivity and complete join semidistributivity are equivalent
for finite lattices. As usual for us, in the infinite setting, the version
which we prefer is the complete one.

Complete join semidistributivity is equivalent to saying that, given
$x,z\in L$, if we consider $\{y\mid x\vee y =z\}$, then this set,
if it is non-empty, has a minimum element. When we say ``minimum element,''
we do not mean only ``minimal'' (i.e., an element such that there is no
element strictly below it), we mean an element which is weakly below all
the elements in the set.

Similarly, a lattice is called meet semidistributive if
$x\wedge y=x\wedge y'$ implies that $x\wedge (y\vee y')$ is also equal to both of them.
It is called completely meet semidistributive if given $x\in L$ and a set
$S\subseteq L$, such that $x\wedge y=z$ for all $y\in S$, then
$x\wedge \bigvee S=z$. Equivalently, given
$x,z\in L$, if we consider $\{y\mid x\wedge y=z\}$, then this set, if
non-empty, has a maximum element.

A lattice is called semidistributive if it is join semidistributive
and meet semidistributive. It is called completely semidistributive if it
is completely join semidistributive and completely meet semidistributive.

Complete semidistributivity is the property which we are going to focus
on. We are now going to develop some properties of completely semidistributive
lattices. 

\begin{proposition} In any completely join semidistributive lattice $L$,
every cover $y\gtrdot x$ has a unique completely join irreducible element
$j$ such that $x\vee j=y$ and $x\vee j_*=x$.\end{proposition}

\begin{proof} 
Let $S=\{z\mid x\vee z=y\}$. This set is non-empty, since $y\in S$. Thus,
by complete join semidistributivity, it has a minimum element. Call it $j$.

Any $z<j$ satisfies that $x\vee z < y$, and thus that $x\vee z=x$. It
follows that any $z<j$ satisfies that $z\leq x$. Therefore, any $z<j$
satisfies $z\leq x\wedge j$. Since $j\not<x$, we have $x\wedge j < j$.
Thus every element strictly below $j$ is weakly below $x\wedge j<j$. It follows that
$j$ is completely join irreducible, and $j_*=j\wedge x$.

Now suppose that we had some other completely join irreducible element $j'$
such that $x\vee j'=y$ and $x\vee j'_*=x$. Since $j$ is the minimum element
of $S$, we must have $j'>j$. But then $x\geq j'_*\geq j$, which
contradicts $x\vee j> x$. Thus $j$ is unique. \end{proof}

Write $\gamma(y\gtrdot x)$ for the completely join irreducible element
defined in the previous proposition.

Similarly, in a completely meet semidistributive lattice $L$, every
cover $y\gtrdot x$ has a unique completely meet irreducible element $m$ such
that $m\wedge y=x$ and $m^*\wedge y=y$. Write $\mu(y\gtrdot x)$ for this
completely meet irreducible element. 

\begin{proposition} In a completely semidistributive lattice $L$, there
are inverse bijections $\kappa$ and $\kappa^d$:

$$\begin{tikzpicture}
\node (a) at (0,0) {$\cji(L)$};
\node (b) at (3,0) {$\cmi(L)$};
\draw[-stealth] ([yshift=.5mm]a.east) -- node[above] {$\kappa$} ([yshift=.5mm]b.west);
\draw[-stealth] ([yshift=-.5mm]b.west) -- node[below] {$\kappa^d$} ([yshift=-.5mm]a.east);
\end{tikzpicture}$$
such that $\kappa(j)=\mu(j\gtrdot j_*)$ and $\kappa^d(m)=\gamma(m^*\gtrdot m)$.
\end{proposition}
It is standard to call these two maps $\kappa$ and $\kappa^d$ but
different sources disagree as to which is which.

\begin{proof} 
Let $j$ be a completely join irreducible element of $L$, and let $m=\kappa(j)=\mu(j\gtrdot j_*)$. We therefore have the following diagram:

$$
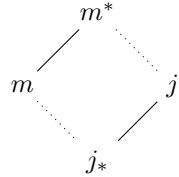
\begin{tikzpicture} \node (a) at (0,0) {$m^*$};
\node (b) at (1,-1) {$j$};
\node (c) at (-1,-1) {$m$};
\node (d) at (0,-2) {$j_*$};
\draw[dotted] (a) -- (b);
\draw[dotted] (c) -- (d);
\draw (a)--(c);
\draw (b) --(d);
\end{tikzpicture}
$$

But now it is clear that $\kappa^d(m^*\gtrdot m)=j$, so $\kappa^d\circ \kappa$
is the identity. The dual argument shows that $\kappa\circ\kappa^d$ is
the identity, and we have shown that $\kappa$ and $\kappa^d$ are inverse
bijections.
\end{proof}

We now have the following theorem, which shows that the two labellings of
the covers of $L$
differ only by a bijection.

\begin{thm} Let $L$ be a completely semidistributive lattice. Then
$\mu(y\gtrdot x)=\kappa(\gamma(y\gtrdot x))$
\end{thm}

\begin{proof} For any $y\gtrdot x$, let $j=\gamma(y\gtrdot x)$ and
$m=\mu(y\gtrdot x)$. We therefore have the following diagram, from which
the result follows.

$$
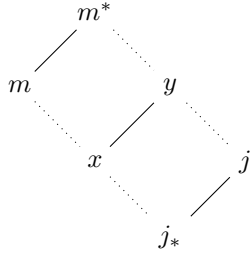
\begin{tikzpicture} \node (a) at (0,0) {$m^*$};
\node (b) at (1,-1) {$y$};
\node (c) at (-1,-1) {$m$};
\node (d) at (0,-2) {$x$};
\node (e) at (2,-2) {$j$};
\node (f) at (1,-3) {$j_*$};
\draw[dotted] (a) -- (b);
\draw[dotted] (c) -- (d);
\draw[dotted] (b) -- (e);
\draw [dotted] (d) -- (f);
\draw (a)--(c);
\draw (b) --(d);
\draw (e) -- (f);
\end{tikzpicture}$$
\end{proof}

\section{Complete semidistributivity of $\tors A$}

The fact that lattices of torsion classes are semidistributive was first
proved by Garver and McConville \cite{GM}. For
not necessarily finite lattices of torsion classes, it turns out to be
natural to consider complete semidistributivity. 

\begin{thm}[{\cite[Theorem 3.1(a)]{DIRRT}}] $\tors A$ is completely semidistributive. \end{thm}

\begin{proof} We will prove complete meet semidistributivity. Complete join
  semidistributivity follows from the complete meet semidistributivity of
  $\tf A$, which is established by a dual argument.
  
Let $\X\in \tors A$, and let $S\subseteq \tors A$ such that for all 
$\Y \in S$, we have $\X\wedge \Y$ is equal. Let $\Z$ be their
common value. Since the meet of torsion classes is intersection,
we have that $\Z=\X\cap \Y$ for any $\Y \in S$.

We want to show that $\X \cap \bigvee S = \Z$ also. 

Clearly $\X \cap \bigvee S \geq \Z$. To prove the opposite inclusion, let
$M\in X\cap \bigvee S$  be a minimal-dimensional counter-example.

Since $M\in \bigvee S$, there is a filtration of $M$
$$0=M_0 \subset M_1 \dots \subset M_r=M$$ with $M_i/M_{i-1} \in \Y_i$, with
$\Y_i \in S$.

Consider
the short exact sequence:
$$ 0 \rightarrow M_1 \rightarrow M \rightarrow M/M_1 \rightarrow 0$$

Now $M/M_1\in X \cap \bigvee S$ since $M$ is. Since $M_1$ is non-zero,
the dimension of $M/M_1$ is less than that of $M$, and thus by our choice
of $M$, we know that $M/M_1$ is not a counter-example. Therefore,
$M/M_1\in \Z$, so in particular $M/M_1\in \Y_1$. On the other hand, we also
know that $M_1\in \Y_1$. Because $\Y_1$ is a torsion class, and therefore
closed under extensions, $M\in \Y_1$. We also know $M\in \X$. Therefore
$M\in \X \cap \Y_1=\Z$. This contradicts our choice of $M$, so it must be
that $\X \cap \bigvee S = \Z$. \end{proof}

\section{Consequences of the complete semidistributivity of $\tors A$}
\label{cons}

As we showed in Section \ref{csd}, a completely semidistributive lattice
has a labelling of every cover relation $y \gtrdot x$
by a completely join irreducible element $\gamma(y\gtrdot x)$,
and a labelling of every cover relation by a completely meet irreducible element
$\mu(y\gtrdot x)$, and
these two labellings are related by the maps $\kappa$ and $\kappa^d$.
We would like to understand what this means in the case of the lattice of
torsion classes.

Since we know that the completely join irreducible torsion classes
correspond to bricks by Theorem \ref{cjib}, for $\Y \gtrdot \X$ in $\tors A$,
define $\hat \gamma(\Y \gtrdot X)= B$, such that $\gamma(\Y\gtrdot \X)=T(B)$.
The following proposition defines $\hat\gamma(\Y\gtrdot \X)$ directly.

\begin{proposition} $\hat\gamma(\Y\gtrdot \X)$ is the unique brick $B$
which is contained in $\Y$ but not in $\X$. \end{proposition}

\begin{proof} By the complete semidistributivity of $\tors A$, we know
that there is a unique completely join irreducible torsion class,
$\gamma(\Y\gtrdot \X)$, such that $\Y\geq \gamma(\Y\gtrdot \X)$ but
$\X \not \geq \gamma(\Y\gtrdot \X)$. By Theorem \ref{cjib}, the
completely join irreducible elements are of the form $T(B)$, for $B$ a
brick.
We have that $\Y \supseteq T(B)$ and $\X \not \supseteq T(B)$ iff
$B\in\Y$ and $B\not\in \X$. So there is a unique brick contained in
$\Y$ but not in $\X$, and it is $\hat\gamma(\Y\gtrdot \X)$.
\end{proof}

Dually, $\mu(\Y\gtrdot\X)=\uperp F(\hat\gamma(\Y\gtrdot\X))$.

\begin{example}[Type $A_2$]
  The brick labelling of the covers of $\tors kQ$ for $Q= 1\leftarrow 2$
  is as follows:

$$
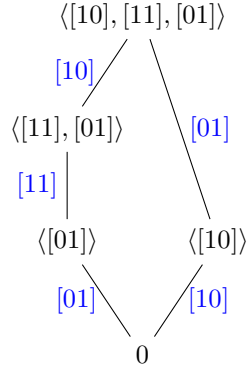
\begin{tikzpicture}[yscale=1.5]
  \node (a) at (0,0) {$\langle [10],[11],[01] \rangle$};
  \node (b) at (-1,-1) {$\langle [11],[01]\rangle$};
  \node (c) at (-1,-2) {$\langle [01]\rangle$};
  \node (d) at (0,-3) {$0$};
  \node (e) at (1,-2) {$\langle [10]\rangle$};

  \draw (a) -- node[left] {\color{blue} $[10]$}(b) --
  node[left] {\color{blue} $[11]$} (c) --
  node[left] {\color{blue} $[01]$} (d) --
  node[right] {\color{blue} $[10]$} (e) --
  node [right] {\color{blue} $[01]$}  (a);
\end{tikzpicture}$$
\end{example}
  
\begin{example}[Kronecker quiver] We revisit the Kronecker quiver from Example \ref{ex3}. The brick labels of some of the covers were already shown in
  Figure \ref{figa}. Inside the interval that is isomorphic to a Boolean
  lattice on the set of tubes, one torsion class covers another
  if they differ exactly in that there is one tube present in one
  but not the other. In this case the brick labelling the cover relation
  is the quasi-simple at the bottom of that tube. 
\end{example}
  
\section{Algebra quotients and lattice quotients}

A surjective map of lattices $L\onto L'$ is called a (complete) lattice quotient if
it respects the (complete) lattice operations.

  For $I$ an ideal of an algebra $A$, consider the algebra quotient $\phi:A\onto A/I$. We can view $\mod A/I$ as the
  subcategory of $\mod A$ consisting of modules annihilated by $I$. We will
  be interested in the map sending $\T$ in $\mod A$ to $\T\cap \mod A/I$. 

  \begin{proposition} $\T\cap \mod A/I$ is a torsion class for $A/I$.
  \end{proposition}
  \begin{proof} It is easy to check that it satisfies the two defining
    conditions. \end{proof}


\begin{proposition}[{\cite[Proposition 5.7(a)]{DIRRT}}] \label{dprop} If $(\T,\F)$ is a torsion pair of $\mod A$, then
  $$(\T\cap \mod A/I,\F\cap \mod A/I)$$ is a torsion pair of $\mod A/I$.
\end{proposition}

\begin{proof} In this proof, when we write $\mathcal C^\perp$ or
  $\uperp \mathcal C$, we always intend it in the ambient category $\mod A$.
  
  Consider $(\T\cap \mod A/I)^\perp$. Clearly this contains $\F$. Now suppose we have some module $M\in \mod A/I$, $M\not\in \F$. There is therefore some $N\in \T$ and some non-zero 
  $f\in \Hom(N,M)\ne 0$. Since $IM=0$, we must have $f(IN)=0$, so $f$
  descends to a map in $\Hom(N/IN,M)$. But $N/IN\in 
  (\T\cap \mod A/I).$ This shows that in fact $M\not\in (\T\cap \mod A/I)^\perp$. We conclude that the torsion free class in $\mod A/I$ which corresponds to
  $\T\cap \mod A/I$ is $\F\cap \mod A/I$.
  \end{proof}

For $\T$ a torsion class in $\mod A$, write $\phibar(\T)$ for $\T\cap \mod A/I$.

\begin{proposition}[{\cite[Proposition 5.7(d)]{DIRRT}}] If $\phi$ is the quotient $A \onto A/I$, then
  $\phibar$ is a lattice quotient from $\tors A$ to $\tors A/I$.
\end{proposition}

\begin{proof} From the definition, it is clear that $\phibar$ respects the
  meet operation on $\tors A$. To see that $\phibar$ respects join, we
  recall that
  $$\bigvee_{\T\in S} \T= \uperp \left(\bigcap_{\T \in S} \T^\perp\right)$$
%
  and the result now follows from Proposition \ref{dprop}.
  \end{proof} 

We are interested in understanding this lattice quotient better. In
particular, we will address the question of when two torsion classes
in $\mod A$ have the same image under this quotient. For this purpose,
we need the following lemma.

\begin{lemma}\label{three} For $\U \leq \V$ in $\tors A$, the following are
  equivalent:\begin{enumerate}
  \item $\U < \V$,
  \item $\U^\perp \cap \V \ne \{0\}$,
  \item $\U^\perp \cap \V$ contains a brick.\end{enumerate}
\end{lemma}
\begin{proof} The implications (3) implies (2) and (2) implies (1) are obvious.

  To see that (1) implies (2), let $X\in \V \setminus \U$. We have a
  short exact sequence
  $$ 0 \rightarrow t_\U X \rightarrow X \rightarrow X/t_\U X \rightarrow 0.$$
  Since $X\not\in \U$, we know that $t_\U X \ne X$, so $X/t_\U X$ is a non-zero module in $\U^\perp$.
  On the other hand, $X\in \V$, so $X/t_\U X$ is also. Thus $X/t_\U X$
  witnesses (2).

  We now show that (2) implies (3). Suppose that $X\in \U^\perp \cap \V$,
  and suppose that the dimension of $X$ is minimal among non-zero modules
  in $\U^\perp \cap \V$. 
  If $X$ is a brick, we are done, so suppose that $X$ is not a brick.
  It therefore has a non-invertible non-zero endomorphism $f$. Let
  $Y=f(X)$. Now $Y$ is at the same time a quotient and a submodule of $X$.
  Since $Y$ is a quotient of $X$, we know that $Y\in \V$. On the other hand,
  since $Y$ is a submodule of $X$, we know that $Y\in \U^\perp$. Therefore
  $Y$ is an element of $\U^\perp \cap \V$ of dimension smaller than $X$,
  contradicting our choice of $X$. Thus $X$ must have been a brick.
  \end{proof}

From Lemma \ref{three}, the following proposition is immediate:

\begin{proposition}[{\cite[Theorem 5.15(b)]{DIRRT}}]\label{equal} For $\U\leq \V$ in $\tors A$, $\phibar (\U)=\phibar (\V)$
  if and only if $\U^\perp\cap\V$ contains no modules annihilated by $I$,
  or equivalently
  contains no bricks annihilated by $I$.
  \end{proposition}

Another way to formulate the proposition is that if $\U\leq \V$, then
$\phibar (\U) \ne \phibar (\V)$ precisely if there is some module in
$\U^\perp \cap \V$ which is annihilated by $I$.

Also, we have the following proposition. We write $\hat\gamma_A$ and
$\hat\gamma_{A/I}$ for the labellings associated to covers in
$\tors A$ and $\tors A/I$, respectively. 

\begin{proposition}[{\cite[Theorem 5.15(a)]{DIRRT}}] \label{agree} If $\Y \gtrdot \X$ in $\tors A$ and
  $\phibar(\Y) \gtrdot \phibar(\X)$ in $\tors A/I$, then 
  $\hat\gamma_A(\Y\gtrdot \X)=\hat\gamma_{A/I}(\phibar(\Y)\gtrdot \phibar(\X))$.
\end{proposition}

\begin{proof} If $\Y\gtrdot\X$ in $\tors A$, then there is a unique brick
  from $\mod A$ in  $\X^\perp \cap \Y$, namely $\hat\gamma_A(\Y\gtrdot\X)$.
  Given that
  $\phibar(\Y)\ne \phibar(\X)$, this brick must in fact lie in
  $\mod A/I$. It is therefore the unique brick in $
  \phibar(\X)^\perp \cap \phibar(\Y)$, and thus equals $\hat\gamma_{A/I}(\Y\gtrdot\X)$.
\end{proof}

\begin{example}[Type $A_2$] Let $A=kQ$, where $Q=1\leftarrow 2$. Let
  $I$ be the ideal of $A$ generated by the arrow. $A/I$ is the path algebra
  of two vertices and no arrows; $\tors \mod A/I$ is as follows:

  $$  
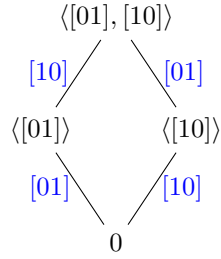
\begin{tikzpicture}[yscale=1.5] \node (a) at (0,0) {$\langle[01],[10]\rangle$};
    \node (b) at (-1,-1) {$\langle [01]\rangle$};
    \node (c) at (1,-1) {$\langle [10]\rangle$};
    \node (d) at (0,-2) {$0$};
    \draw (a) -- node [left] {\color{blue}$[10]$} (b) -- node [left] {\color{blue}$[01]$} (d) -- node[right] {\color{blue} $[10]$} (c) -- node[right] {\color{blue} $[01]$} (a);
  \end{tikzpicture} $$

  We see that it is obtained from the lattice $\tors A$ by identifying
  the two torsion classes $\langle[01],[11]\rangle$ and $\langle[01]\rangle$,
  which differ only in modules which are not in $\mod A/I$. We further see
  that the labels of the cover relations which remain cover relations in
  $\tors A/I$ receive the same labels as cover relations in
  $\tors A$ and as cover relations in $\tors A/I$, consistent with
  Proposition \ref{agree}.
  \end{example}

In the next section, we will see how to combine Proposition \ref{equal}  with the
labelling $\hat\gamma$. In order to do that, we need another important
structural result about $\tors A$.

\section{$\tors A$ is weakly atomic}
 
A lattice is called weakly atomic if in any interval $[u,v]$ with $u<v$,
there is some pair of elements $x,y$ with $x\lessdot y$. (This property
is referred to as arrow-separatedness in the current version of \cite{DIRRT} and
as cover-separatedness in the current version of \cite{RST}, but they will be
updated to reflect the standard terminology.)
The interval $[0,1]$ in $\mathbb R$, with the usual order, is an example
of a lattice which is not weakly atomic (since it has no cover relations
at all).

In this section, we will prove the following two theorems.

\begin{thm}[\cite{DIRRT}]\label{tha} $\tors A$ is weakly atomic. \end{thm}

\begin{thm}[\cite{DIRRT}]\label{thb} Let $\phi:A\onto A/I$ be an algebra quotient. For
  $\U\subseteq \V$, we have that $\phibar(\U)=\phibar(\V)$ iff all covers in
  $[\U,\V]$ are labelled by bricks which are not annihilated by $I$.\end{thm}

On the way to proving these theorems, we first prove
the following proposition, which can be viewed as a relative
version of Theorem \ref{cjib}.

\begin{proposition}[{\cite[Theorem 3.4]{DIRRT}}] \label{rcjib} Let $\U\leq \V$ be two torsion classes.
  The map 
  $B\to T(B)\vee \U$ is a bijection from $\br \!(\U^\perp \cap \V)$
  to $\cji [\U,\V]$.
  \end{proposition}

\begin{proof} Let $B \in \br \!(\U^\perp \cap \V)$. Let $\Y=\U\vee T(B)$.
  Also consider the torsion class $\X=\Y\cap \uperp F(B)$.
  Since $B\in \U^\perp$,
  we have that $\uperp F(B) \supseteq \U$, so $\X$ also lies in $[\U,\V]$. 
  Now $\X$ is strictly
  contained in $\Y$ since it does not contain $B$. But any torsion class
  containing $\U$ which is strictly contained in $\Y$ cannot include any
  module admitting a surjective map onto $B$. By Lemma \ref{slem}, any such
  torsion class is therefore contained in $\uperp B=\uperp F(B)$. This shows
  that $\Y$ is completely join irreducible in
  $[\U,\V]$ and that $\Y$ covers $\X$.

  We now show that all the completely join irreducible elements of $[\U,\V]$
  correspond to some brick as above.
  Any torsion class can be written as the join of the torsion classes
  corresponding to the bricks it contains, so any torsion class
  in $[\U,\V]$ can be written as the join of $\U$ and a set of torsion classes
  of the form $T(B)$ for $B$ lying in some subset of $\U^\perp \cap \V$. It
  follows that the only completely join irreducible elements of
  $[\U,\V]$ are those of the form $\U\vee T(B)$.

  Finally, the map from bricks to torsion classes is invertible. If
  $\Y$ is a completely join irreducible torsion class in
  $[\U,\V]$, with $\X$ the unique torsion class in $[\U,\V]$ which it
  covers, then the brick corresponding to $\Y$ is $\hat\gamma(\Y\gtrdot \X)$.
  \end{proof}

Based on this, we can now easily establish the following proposition:

\begin{proposition}[\cite{DIRRT}]\label{interval} Let $\U < \V$ be two torsion classes in $\mod A$.
  Then there are covers in $[\U,\V]$ labelled by each brick in
  $\U^\perp \cap \V$, and no others. \end{proposition}

Note that $\U^\perp \cap \V$ is non-empty by Lemma \ref{three}.

\begin{proof}
  It is clear that no other brick can appear as a label since if
  $\V\geq \Y\gtrdot \X\geq \U$, and 
  $\hat\gamma(\Y\gtrdot\X)=B$ then $B\in \Y\subseteq \V$ and
  $B\in \X^\perp \subseteq \U^\perp$.

  For the converse direction, if $B$ is a brick in $\U^\perp\cap \V$, then
  by Proposition \ref{rcjib}, there is a completely join irreducible
  in $[\U,\V]$ corresponding to $B$, and the cover relation down from it
  in $[\U,\V]$ is labelled by $B$.
  \end{proof}
  


Theorem \ref{tha} follows directly from Proposition \ref{interval}, since
if $\U<\V$, then by Lemma \ref{three}, $\U^\perp\cap \V$ contains a brick.

Theorem \ref{thb} follows as well, by combining Proposition \ref{equal} with Proposition \ref{interval}.

  \section{A combinatorial application: finite semi-distributive lattices}

  Consider the following lattice:
  $$
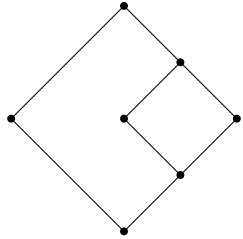
\begin{tikzpicture}[scale=.75] \coordinate (a) at (0,0);
    \coordinate (b) at (1,-1);
    \coordinate (c) at (2,-2);
    \coordinate (d) at (0,-2);
    \coordinate (e) at (1,-3);
    \coordinate (f) at (0,-4);
    \coordinate (g) at (-2,-2);

    \draw (a) -- (c) -- (f) -- (g) -- (a);
    \draw (b) -- (d) -- (e);
    \foreach \x in {(a), (b), (c), (d), (e), (f), (g)}{
        \fill \x circle[radius=2pt];}
\end{tikzpicture}$$
  This lattice is semidistributive. Suppose that it were isomorphic to
  $\tors A$ for some $A$. 
  We see that this lattice has four
  (completely) join irreducible elements and four (completely) meet irreducible
  elements, so $\mod A$ would necessarily have four bricks by Theorem \ref{cjib}. We see that two
  of the bricks would have to be simple, call them $S_1$ and $S_2$,
  and there would be maps as follows, with $X$ and $Y$ being the other two
  bricks:

  $$\begin{tikzpicture}
    \node (w) at (0,0) {$S_1$};
    \node (x) at (1.5,.5) {$X$};
    \node (y) at (1.5,-.5) {$Y$};
    \node (z) at (3,0) {$S_2$};
    \draw[right hook->] (w) -- (x);
    \draw [right hook->] (w) -- (y);
    \draw [->>] (x) -- (z);
    \draw [->>] (y) -- (z);
\end{tikzpicture}$$
  There is no such module category. Results of \cite{AP}, extending \cite{J},
  can also be used to construct many examples of finite semidistributive
  lattices which are not lattices of torsion classes.

  In light of this, it would seem unlikely that representation theory could help us
  to understand general finite semidistributive lattices. Nonetheless, it turns out
  that it can. Indeed, in \cite{RST}, inspired by properties of lattices of torsion classes,
  we gave a construction which yields exactly the finite semidistributive lattices. I will close by
  describing this construction.
  
  
  
  Given any finite set, which we will call $\Br$, and a reflexive relation
  $\sto$ on $\Br$, for a subset $\mathcal C\subset \Br$, we can define $\mathcal C^\perp=
  \{ Y \in \Br \mid \forall X\in \mathcal C, X \not\sto Y\}$, and
  ${}^\perp\mathcal C=\{X \mid \forall Y \in \mathcal C, X \not\sto Y\}$.
  Torsion pairs in $\Br$ are then defined to be pairs of subsets
  $(\mathcal T,\mathcal F)$ such that $\mathcal T^\perp=\mathcal F$ and
  $\mathcal T= {}^\perp\mathcal F$. We can then define $\torss (\Br,\sto)$
  to be the set of torsion pairs $(\mathcal T,\mathcal F)$, ordered by inclusion on
  $\T$.
    If we allow ourselves to start with any set $\Br$ and reflexive relation $\sto$,
  this construction is so general as to be able to construct any
  finite lattice, as was discovered by Markowsky \cite{Ma}.

Therefore, if we want to get only semidistributive lattices, we need to put some further conditions on
$\rightarrow$. It turns out that the way to do this is the insist that the relation $\rightarrow$ on the set $\Br$ be more like the relation ``there exists a non-zero morphism'' on the set of bricks of a module category. 

We make this precise as follows.
  Starting from a reflexive relation
  $\sto$, define two other relations, $\onto$ and $\into$. We define
  $X\onto Y$ iff whenever $Y\sto Z$ then $X\sto Z$. Similarly, we define
  $X\into Y$ iff whenever $Y\sto Z$ then  $X\sto Z$.
  Again, the intuition from representation theory is clear:  
  if $M$ and $N$
  are $A$-modules and there is a surjection from
  $M$ to $N$ then whenever there is
  a non-zero map from $N$ to some $L$, then there is also a non-zero map
  from $M$ to $L$ and dually for injections. (Note, though, that if we
  take $\Br=\br\mod A$ and take $\sto$ to be ``there exists a non-zero
  morphism'', the relations $\onto$ and $\into$ defined as above are not exactly
  ``there exists a surjection'' and ``there exists an injection''. See
  \cite[Section 8]{RST} for more details.)

  We say that a reflexive
  relation $\sto$ on $\Br$ is factorizable if it satisfies
  the following two conditions:
  \begin{itemize}
  \item For any $X,Z \in \Br$ with $X\sto Z$, there exists $Y\in \Br$ such that
    $$X\onto Y \into Z.$$
  \item Any of
    $X\onto Y \onto X$ or $X\into Y \onto X$, or $X \into Y \into X$ imply
  $X=Y$.
    \end{itemize}
  As is probably clear, the
  motivating intuition for the first condition is that a non-zero
  morphism can be factored as a surjection followed by an injection.

  We can now state the main result of \cite{RST}:

  \begin{thm}[{\cite[Theorem 1.2]{RST}}] Let $\Br$ be a finite set, and $\sto$ a reflexive
    factorizable relation on $\Br$. Then $\torss (\Br,\sto)$ is a
    semidistributive lattice, and every semidistributive lattice arises in
    this way for a choice of $\Br$ and $\sto$ which is unique up to
    isomorphism. \end{thm}

  I close with the following question:

  \begin{question} Is there a way to interpret any finite semidistributive
    lattice as
    the lattice of torsion classes of a ``real'' category?
  \end{question}

  The question is deliberately worded somewhat imprecisely. Another way to
  ask the question would be to ask for a representation-theoretic meaning
  to the construction of finite semidistributive lattices of \cite{RST}.

  \section*{Acknowledgements} I would like to thank my coauthors on
  \cite{IRRTT,DIRRT,RST}, from whom I have  learned a great deal.
  It is my pleasure to acknowledge NSERC and the Canada Research Chairs program
  for their financial support. Thanks to Nathan Reading, Alexander Garver, and two anonymous referees for helpful comments
  on this manuscript.
  I am extremely grateful to have been given
  the opportunity
  to present this material at Zhejiang University in 2018 and at
  the Isfahan School on Representations
  of Algebras in 2019. Thanks to Fang Li for the invitation to
  Zhejiang and to the organizers of the Isfahan School and Conference on
  Representations of Algebras, and in particular Javad Asadollahi, for
  the invitation to speak in Isfahan, and for the invitation to prepare
  this contribution to the special issue.

  \end{document}